\documentclass[12pt,reqno]{amsart}
\usepackage{amsmath, amsthm, amssymb}

\advance \topmargin by -\headheight
\advance \topmargin by -\headsep
     
\evensidemargin \oddsidemargin
\newtheorem{theorem}{Theorem}[section]
\newtheorem{lemma}[theorem]{Lemma}

\theoremstyle{definition}

\theoremstyle{remark}

\numberwithin{equation}{section}

\def\dbN{{\mathbb N}}
\def\dbR{{\mathbb R}}
\def\dbZ{{\mathbb Z}}

\def\grB{{\mathfrak B}}
\def\grc{{\mathfrak c}}\def\grC{{\mathfrak C}}
\def\grD{{\mathfrak D}}

\def\grm{{\mathfrak m}}\def\grM{{\mathfrak M}}\def\grN{{\mathfrak N}}
\def\grn{{\mathfrak n}}
\def\grB{{\mathfrak B}}\def\grC{{\mathfrak C}}
\def\grK{{\mathfrak K}}

\def\alp{{\alpha}}  

\def\bet{{\beta}}

\def\Gam{{\Gamma}}
\def\del{{\delta}} \def\Del{{\Delta}}

\def\tet{{\theta}}

\def\sig{{\sigma}}  

\def\Ups{{\Upsilon}}

\def\d{{\partial}}
\def\eps{\varepsilon}

\def\le{\leqslant} \def\ge{\geqslant}

\def\d{{\,{\rm d}}}

\begin{document}
\title[Sums of two like powers]{Additive representation in short intervals, II:\\ Sums of two like powers}
\author[J\"org Br\"udern]{J\"org Br\"udern}
\address{Mathematisches Institut, Bunsenstrasse 3--5, D-37073 G\"ottingen, Germany}
\email{bruedern@uni-math.gwdg.de}
\author[Trevor D. Wooley]{Trevor D. Wooley}
\address{School of Mathematics, University of Bristol, University Walk, Clifton, Bristol BS8 1TW, United 
Kingdom}
\email{matdw@bristol.ac.uk}
\thanks{The first author acknowledges support by Deutsche Forschungsgemeinschaft. 
The second author is grateful for the support and excellent working conditions 
provided at Mathematisches Institut, G\"ottingen, through the Gauss Professorship of 
Akademie der Wissenschaften zu G\"ottingen, which greatly facilitated the preparation 
of this paper.}
\date{}
\begin{abstract} We establish that, for almost all natural numbers $N$, there is a sum of 
two positive integral cubes lying in the interval $[N-N^{7/18+\eps},N]$. Here, the 
exponent $7/18$ lies half way between the trivial exponent $4/9$ stemming from the 
greedy algorithm, and the exponent $1/3$ constrained by the number of integers 
not exceeding $X$ that can be represented as the sum of two positive integral cubes. We 
also provide analogous conclusions for sums of two positive integral $k$-th powers when 
$k\ge 4$. 
\end{abstract}
\maketitle

\section{Introduction} The sequence of integers $2=s_{k,1}<s_{k,2}<\ldots$ represented 
as the sum of two $k$-th powers of natural numbers is certainly sparse when $k\ge 3$, 
for a simple counting argument confirms that their number, $\nu_k(N)$, not exceeding 
$N$ is at most $O(N^{2/k})$. Investigations concerning $\nu_k(N)$ date at least as far 
back as the work of Erd\H os and Mahler \cite{Erd1939, EM1938}, which showed that 
$\nu_k(N)\gg N^{2/k}$. Hooley \cite{Hoo1963, Hoo1964, Hoo1981a, Hoo1981b, 
Hoo1996} has returned to the problem on numerous occasions, and when $h\ge 3$ has 
established the asymptotic formula
\begin{equation}\label{1.1}
\nu_h(N)=\frac{\Gam(1+1/h)^2}{2\Gam (1+2/h)}N^{2/h}+O(N^{5/(3h)+\eps}).
\end{equation}
This conclusion derives from the paucity of numbers that are represented as the sum of 
two $h$-th powers in two essentially distinct ways. Other scholars have augmented and 
refined Hooley's opera (see Greaves \cite{Gre1966, Gre1994}, Skinner and Wooley 
\cite{SW1995}, Wooley \cite{Woo1995}, Heath-Brown \cite{HB1997, HB2002}, Browning 
\cite{Bro2002}, Salberger \cite{Sal2008}). The distribution of such numbers in short 
intervals has, thus far, received little attention, although Daniel \cite{Dan1997} has 
considered the corresponding problem for sums of three positive integral cubes. In this 
memoir we remedy this situation.\par

Given a large integer $n$, one may subtract from $n$ the largest integral $k$-th power 
not exceeding $n$, leaving a remainder of size at most $kn^{1-1/k}$. By repeating this 
greedy algorithm, one finds that for all large $N$, there is a sum of two positive integral 
$k$-th powers 
between $N-k^2N^{\phi_k}$ and $N$, where $\phi_k=(1-1/k)^2$. The main result of 
this paper shows that the same conclusion remains valid, with a smaller exponent in place 
of $\phi_k$, for {\it almost all} natural numbers $N$. Denote by $E_k(N,Z)$ the number 
of natural numbers $N<n\le 2N$ for which the interval $(n,n+Z]$ contains no integer that 
is the sum of two positive integral $k$-th powers. When $k\ge 3$, we put
\begin{equation}\label{1.a}
\sig_k=\begin{cases} 2^{2-k},&\text{when $3\le k\le 7$,}\\
(2k^2-10k+12)^{-1},&\text{when $k\ge 8$,}\end{cases}
\end{equation}
and define
\begin{equation}\label{1.b}
\tet_k=1-\frac{2}{k}+\frac{1-\sig_k}{k^2}=\phi_k-\frac{\sig_k}{k^2}.
\end{equation}

\begin{theorem}\label{theorem1.1}
Suppose that $k\ge 3$. Then, whenever $Z\ge N^{\tet_k}$, one has
\begin{equation}\label{1.2}
E_k(N,Z)\ll N^{1+\tet_k+\eps}Z^{-1}.
\end{equation}
\end{theorem}

Whereas the greedy algorithm ensures that $E_k(N,2k^2N^{\phi_k})\ll 1$, the conclusion 
of Theorem \ref{theorem1.1} yields the bound $E_k(N,N^{\phi_k-\del})=o(N)$ whenever 
$\del<\sig_k/k^2$. The spacing of sums of two $k$-th powers evident in the asymptotic 
formula (\ref{1.1}), meanwhile, implies that $E_k(N,Z)\gg N$ whenever $Z\le N^{1-2/k}$. 
It seems plausible that (\ref{1.2}) should remain valid provided only that $\tet_k>1-2/k$. 
Our estimate is particularly strong in the case $k=3$, where we show that for all 
$\eps>0$, and almost all $N\in \dbN$, there is a sum of two positive integral cubes lying 
between $N$ and $N+N^{7/18+\eps}$. Here, the exponent $7/18$ lies half way between 
the trivial exponent $4/9$ stemming from the greedy algorithm, and the exponent $1/3$ 
constrained by the asymptotic formula (\ref{1.1}).\par

The conclusion of Theorem \ref{theorem1.1} also delivers bounds for the size of the gaps 
between sums of two $k$-th powers in mean square.

\begin{theorem}\label{theorem1.2} When $k\ge 3$, one has
$$\sum_{s_{k,n}\le N}(s_{k,n+1}-s_{k,n})^2\ll N^{1+\tet_k+\eps}.$$
\end{theorem}

We note in particular that since (\ref{1.1}) shows that, for almost all $n\in \dbN$, one 
has $s_{k,n+1}-s_{k,n}\gg s_{k,n}^{1-2/k}$, then
$$\sum_{N/2<s_{k,n}\le N}(s_{k,n+1}-s_{k,n})^2\gg 
(N^{1-2/k})^2N^{2/k}=N^{2-2/k}.$$
This lower bound is expected to reflect the asymptotic behaviour of the mean square gap 
size estimated in Theorem \ref{theorem1.2}. Meanwhile, the bound 
\begin{equation}\label{1.4a}
s_{k,n+1}-s_{k,n}\ll s_{k,n}^{\phi_k},
\end{equation}
immediate from the greedy algorithm, yields the estimate
$$\sum_{s_{k,n}\le N}(s_{k,n+1}-s_{k,n})^2\ll N^{\phi_k}\sum_{s_{k,n}\le N}
(s_{k,n+1}-s_{k,n})\ll N^{1+\phi_k}.$$
In view of (\ref{1.b}), one has $2-2/k<1+\tet_k<1+\phi_k$, so that the conclusion of 
Theorem \ref{theorem1.2} improves on the trivial estimate, but falls short of the 
aforementioned expectation. In the case $k=3$, the exponent $1+\tet_3=25/18$ lies half 
way between the trivial and conjectured bounds.\par

In the above discussion, we have deliberately restricted attention to the situation in which 
$k\ge 3$. The behaviour of the sequence $(s_{2,n})$, consisting of sums of two squares, 
is quite different. We refer the reader to Friedlander \cite{Fri1982}, Harman 
\cite{Har1991}, Hooley \cite{Hoo1994} and Plaksin \cite{Pla1987, Pla1992} for a 
consideration of the distribution of gaps in this relatively dense sequence.\par

The exceptional set estimate presented in Theorem \ref{theorem1.1} is obtained by 
applying the Hardy-Littlewood (circle) method to the Diophantine equation
\begin{equation}\label{1.3}
x^k+y^k+z=n,
\end{equation}
with $z$ running over a short interval. By applying Bessel's inequality, one is led to 
consider a mean value estimate implicitly related to the number of integral solutions of the 
equation
\begin{equation}\label{1.4}
x_1^k-x_2^k=y_1^k-y_2^k+z_1-z_2,
\end{equation}
with $x_i$ and $y_i$ bounded above by $n^{1/k}$, and with $z_i$ in the same short 
interval. Aficionados of the circle method will recognise the potential for applying 
arguments based on the use of diminishing ranges, in which the variables $y_i$ are 
constrained to lie in a slightly shortened interval. Two obstacles prevent a pedestrian 
treatment of this problem. First, one must apply diminishing ranges in a treatment 
restricted to minor arcs only. Also, one has the second challenge of handling a problem in 
which the number of variables is very small. Methods pursued in the first of this series of 
papers \cite{BW2004} may be adapted to surmount the first of these difficulties (see also 
\cite{Bru2001} and \cite{Vau1986c} for earlier such treatments). Meanwhile, the second 
may be overcome by solving a long sequence of pruning exercises, all within range of the 
accomplished practitioner of such methods.\par

In this paper, we adopt the convention that whenever $\eps$ appears in a statement, 
either implicitly or explicitly, then the statement holds for each $\eps>0$. Implicit 
constants in the notations of Landau and Vinogradov will depend at most on $\eps$ and 
$k$. Finally, write $\|\tet\|=\underset{y\in\dbZ}{\min}|\tet-y|$ and $e(z)$ for 
$e^{2\pi i z}$.

\section{Infrastructure} We begin by introducing the notation and cast of generating 
functions required to describe our method. We consider a fixed integer $k$ with $k\ge 3$, 
and we define $\sig=\sig_k$ and $\tet=\tet_k$ as in (\ref{1.a}) and (\ref{1.b}). Let $N$ 
be a sufficiently large positive number, and define
\begin{equation}\label{2.1}
X=(N/3)^{1/k},\quad Y=X^{1-(1-\sig)/k},\quad H=2^kX^\sig \quad \text{and}\quad 
Q=X^{1-\sig/k}.
\end{equation}
Also, we consider a real number $Z$ with
\begin{equation}\label{2.2}
X^{k\tet}\le Z\le 6k^2X^{k-2+1/k}.
\end{equation}
Let $r(n;Z)$ be the number of integral solutions of the equation (\ref{1.3}) with 
$X<x\le 2X$, $Y<y\le 2Y$ and $1\le z\le Z$. Our goal is an estimate for the quantity
\begin{equation}\label{2.5}
\Ups (N,Z)=\sum_{N<n\le 2N}\left|r(n;Z)-k^{-1}n^{-1+1/k}YZ\right|^2.
\end{equation}

\par We bound $\Ups(N,Z)$ through the medium of the Hardy-Littlewood method. The 
exponential sums required in this enterprise are
\begin{equation}\label{2.6}
f(\alp)=\sum_{X<x\le 2X}e(\alp x^k),\quad g(\alp)=\sum_{Y<y\le 2Y}e(\alp y^k),\quad 
u(\alp)=\sum_{1\le z\le Z}e(\alp z).
\end{equation}
It will be expedient on numerous occasions to suppress the argument $\alp$ from these 
notations as an aid to exposition and concision. Thus $f(\alp)$ may be abbreviated to $f$, 
for example. By orthogonality, one has
\begin{equation}\label{2.7}
r(n;Z)=\int_0^1f(\alp)g(\alp)u(\alp)e(-n\alp)\d \alp,
\end{equation}
the relation which provides the starting point for our analysis of $\Ups(N,Z)$. With $Q$ 
defined as in (\ref{2.1}), we write $\grM$ for the union of the intervals
$$\grM(q,a)=\{ \alp \in [0,1):|q\alp-a|\le QX^{-k}\},$$
with $0\le a\le q\le Q$ and $(a,q)=1$. Also, we denote by $\grM^\dagger$ the 
corresponding union of the intervals $\grM(q,a)$ in which $q>1$. Further, we put 
$\grm=[0,1)\setminus \grM$. When $\grB\subseteq [0,1)$ is measurable, we write
$$r_\grB (n;Z)=\int_\grB f(\alp)g(\alp)u(\alp)e(-n\alp)\d\alp .$$
Thus, in view of (\ref{2.7}), we have
\begin{equation}\label{2.8}
r(n;Z)=r_\grM(n;Z)+r_\grm(n;Z).
\end{equation}

\par We next introduce the quantities
$$\Ups_\grm=\sum_{N<n\le 2N}|r_\grm(n)|^2\quad \text{and}\quad 
\Ups_\grM=\sum_{N<n\le 2N}|r_\grM(n;Z)-k^{-1}n^{-1+1/k}YZ|^2.$$
Substituting (\ref{2.8}) into (\ref{2.5}), we thus arrive at the estimate
\begin{equation}\label{2.9}
\Ups(N,Z)\le 2(\Ups_\grm+\Ups_\grM).
\end{equation}
We estimate the contribution of $\Ups_\grM$ in \S3, deferring the consideration of 
$\Ups_\grm$ to \S\S4 and 5.

\section{The collapse of the major arcs}
We set about the task of replacing the generating functions $f$ and $u$ by their natural 
major arc approximants. We write
\begin{equation}\label{3.aa}
S(q,a)=\sum_{r=1}^qe(ar^k/q)\quad \text{and}\quad V(\bet;P)=
\sum_{P^k<x\le (2P)^k}k^{-1}x^{-1+1/k}e(\bet x),
\end{equation}
and put $v(\bet)=V(\bet;X)$ and $w(\bet)=V(\bet;Y)$. Next, we define the function 
$f^*(\alp)$ for $\alp \in \grM(q,a)\subseteq \grM$ by putting
$$f^*(\alp)=q^{-1}S(q,a)v(\alp-a/q),$$
and we set $f^*(\alp)=0$ for $\alp\in \grm$. Also, we define
\begin{equation}\label{3.2}
u^*(\alp)=\begin{cases} u(\alp),&\text{when $\|\alp\|\le QX^{-k}$,}\\
0,&\text{otherwise.}
\end{cases}
\end{equation}
We record for future reference an estimate of use in replacing $f(\alp)$ by $f^*(\alp)$ 
when $\alp\in \grM$, with a similar estimate concerning $u(\alp)$ and $u^*(\alp)$.

\begin{lemma}\label{lemma3.1} When $\alp\in \grM$, one has
$$f(\alp)-f^*(\alp)\ll Q^{1/2+\eps}\quad \text{and}\quad u(\alp)-u^*(\alp)\ll Q.$$
\end{lemma}

\begin{proof} The claim concerning $f$ is immediate from \cite[Theorem 4.1]{Vau1997}. 
Meanwhile, from the relation
$$u(a/q)=\sum_{r=1}^qe(ar/q)\left( Z/q+O(1)\right),$$
valid for $a\in \dbZ$ and $q\in \dbN$, it follows via partial summation that
\begin{equation}\label{3.2a}
u(\bet+a/q)=q^{-1}\biggl( \sum_{r=1}^qe(ar/q)\biggr) u(\bet)+O\left( q(1+Z|\bet|)
\right) .
\end{equation}
A similar argument is employed in the proof of \cite[Lemma 2.7]{Vau1997}. When 
$q>1$ and $(a,q)=1$, one has
$$\sum_{r=1}^qe(ar/q)=0.$$
Thus, when $\alp\in \grM(q,a)\subseteq \grM$ with $q>1$, one deduces that
$$u(\alp)\ll q+Z|q\alp-a|\ll Q+ZQX^{-k}\ll Q.$$
When $\alp\in \grM(q,a)\subseteq \grM$ with $q=1$, meanwhile, one has 
$\|\alp\|\le QX^{-k}$, and hence $u(\alp)=u^*(\alp)$. Thus, in any case, we have 
$u(\alp)-u^*(\alp)\ll Q$, and the proof of the lemma is complete.
\end{proof}

We continue with an auxiliary mean value estimate. Write
\begin{equation}\label{3.4}
I_1=\int_0^1|g(\alp)u(\alp)|^2\d\alp .
\end{equation}

\begin{lemma}\label{lemma3.0}
One has $I_1\le YZ$.
\end{lemma}

\begin{proof} By orthogonality, we see that $I_1$ counts the number of integral solutions 
of the equation $y_1^k-y_2^k=z_1-z_2$, with $Y<y_1,y_2\le 2Y$ and 
$1\le z_1,z_2\le Z$. When $y_1\ne y_2$, one has $|y_1^k-y_2^k|\ge kY^{k-1}>Z$. The 
only solutions of this equation counted by $I_1$ consequently satisfy $y_1=y_2$, whence 
$I_1\le YZ$. This completes the proof of the lemma.
\end{proof}

We are now equipped to pursue the replacement process.

\begin{lemma}\label{lemma3.2}
One has
\begin{equation}\label{3.3a}
\int_\grM |(f-f^*)gu|^2\d\alp \ll XYZ\quad \text{and}\quad \int_{\grM^\dagger}
|f^*gu|^2\d\alp \ll X^{1+\eps}YZ.
\end{equation}
\end{lemma}

\begin{proof} An application of Lemma \ref{lemma3.1} leads from (\ref{3.4}) via 
Lemma \ref{lemma3.0} to the estimate
$$\int_\grM |(f-f^*)gu|^2\d\alp \ll Q^{1+\eps}I_1\ll XYZ,$$
confirming the first bound of (\ref{3.3a}).\par

For the second bound we must work harder. Note that, from (\ref{3.2}), one has 
$u^*(\alp)=0$ for $\alp\in \grM^\dagger$. Hence we deduce from 
Lemma \ref{lemma3.1} that
$$\int_{\grM^\dagger}|f^*gu|^2\d\alp \ll Q^{2+\eps}\int_\grM |f^*g|^2\d\alp .$$
An application of H\"older's inequality shows that
$$\int_\grM|f^*g|^2\d\alp \le \biggl( \int_\grM|f^*|^{k+1}\d\alp \biggr)^{2/(k+1)}
\biggl( \int_0^1|g|^4\d\alp \biggr)^{1/2}.$$
The first integral on the right hand side here may be estimated through the methods of 
\cite[Chapter 4]{Vau1997} (see, in particular, \cite[Lemmata 4.9 and 6.2]{Vau1997}), and 
the second integral via Hua's lemma (see \cite[Lemma 2.5]{Vau1997}). Thus
\begin{align}
\int_{\grM^\dagger}|f^*gu|^2\d\alp &\ll Q^{2+\eps}(X^{1+\eps})^{2/(k+1)}
(Y^{2+\eps})^{1/2}\notag \\
&\ll X^{1+2\eps}YZ(Q^2X^{-1+2/(k+1)}Z^{-1}).\label{3.6}
\end{align}
Since $k\ge 3+2/(k+1)-(1+\sig)/k$ when $k\ge 3$, it follows that
$$k-2+(1-\sig)/k\ge 2/(k+1)-1+2(1-\sig/k),$$
so that in view of (\ref{1.b}), (\ref{2.1}) and (\ref{2.2}), the parenthetic factor on the 
right hand side of (\ref{3.6}) is at most $1$. This confirms the second bound of 
(\ref{3.3a}) and completes the proof of the lemma.
\end{proof}

We combine the two estimates of Lemma \ref{lemma3.2} in the next lemma.

\begin{lemma}\label{lemma3.3}
One has
$$\int_{\grM^\dagger}|fgu|^2\d\alp \ll X^{1+\eps}YZ.$$
\end{lemma}

\begin{proof} The elementary inequality $|f|^2\ll |f-f^*|^2+|f^*|^2$ implies that
$$\int_{\grM^\dagger}|fgu|^2\d\alp \ll \int_\grM |(f-f^*)gu|^2\d\alp 
+\int_{\grM^\dagger}|f^*gu|^2\d\alp ,$$
and the desired conclusion is now immediate from Lemma \ref{lemma3.2}.
\end{proof}

We define the central interval $\grC =[-QX^{-k},QX^{-k}]$, and note that
$$r_\grM(n;Z)=r_{\grM^\dagger}(n;Z)+r_\grC(n;Z).$$
It is useful to observe that when $\alp\in \grC$, one has $f^*(\alp)=v(\alp)$. Next, put
$$\rho_1(n;Z)=\int_\grC v(\alp)g(\alp)u(\alp)e(-n\alp)\d\alp .$$
Since $\grC\subseteq \grM+\dbZ$, an application of Bessel's inequality leads us via 
Lemma \ref{lemma3.2} to the bound
\begin{equation}\label{3.8}
\sum_{N<n\le 2N}|r_\grC(n;Z)-\rho_1(n;Z)|^2\le \int_\grC |(f-f^*)gu|^2\d\alp \ll XYZ.
\end{equation}
Likewise, we deduce via Lemma \ref{lemma3.3} that
\begin{equation}\label{3.9}
\sum_{N<n\le 2N}|r_{\grM^\dagger}(n;Z)|^2\le \int_{\grM^\dagger}|fgu|^2\d\alp \ll 
X^{1+\eps}YZ.
\end{equation}

\par The singular integral is
\begin{equation}\label{3.10}
\rho_2(n;Z)=\int_{-1/2}^{1/2}v(\alp)g(\alp)u(\alp)e(-n\alp)\d\alp ,
\end{equation}
and we next compare this expression to $\rho_1(n;Z)$.

\begin{lemma}\label{lemma3.4}
One has
$$\sum_{N<n\le 2N}|\rho_1(n;Z)-\rho_2(n;Z)|^2\ll XYZ.$$
\end{lemma}

\begin{proof} An application of Bessel's inequality conveys us from (\ref{3.10}) via 
\cite[Lemma 6.2]{Vau1997} to the bound
\begin{align*}
\sum_{N<n\le 2N}|\rho_1(n;Z)-\rho_2(n;Z)|^2&\ll \int_{QX^{-k}}^{1/2}
|v(\alp)g(\alp)u(\alp)|^2\d\alp \\
&\ll (XYZ)^2\int_{QX^{-k}}^{1/2}(1+X^k\alp)^{-2}\d\alp .
\end{align*}
Thus we conclude that
$$\sum_{N<n\le 2N}|\rho_1(n;Z)-\rho_2(n;Z)|^2\ll XYZ(X^{1-k}YZQ^{-1}).$$
Since $Q\ge Y$ and $Z\le X^{k-1}$, the parenthetic factor on the right hand side here 
does not exceed $1$, and so the proof of the lemma is complete.
\end{proof}

The singular integral may be evaluated with an error acceptable in mean square.

\begin{lemma}\label{lemma3.5}
One has
$$\sum_{N<n\le 2N}\left|\rho_2(n;Z)-k^{-1}YZn^{-1+1/k}\right|^2\ll XYZ.$$
\end{lemma}

\begin{proof} By orthogonality, it follows from (\ref{3.10}) that
$$\rho_2(n;Z)=k^{-1}\sum_{Y<y\le 2Y}\sum_{1\le z\le Z}\sum_{\substack{
X^k<m\le (2X)^k\\ m+y^k+z=n}}m^{-1+1/k}.$$
Observe that when $n>N$, $y\le 2Y$ and $z\le Z$, one has
$$m=n-y^k-z=n\left( 1+O(HX^{-1}+X^{-2+1/k})\right).$$
Hence
$$m^{-1+1/k}=n^{-1+1/k}(1+O(HX^{-1})),$$
and so it follows that
$$\rho_2(n;Z)=k^{-1}YZn^{-1+1/k}(1+O(HX^{-1})).$$
We thus deduce that
\begin{align*}
\sum_{N<n\le 2N}\left|\rho_2(n;Z)-k^{-1}YZn^{-1+1/k}\right|^2&\ll 
(YZ)^2N^{-1+2/k}H^2X^{-2}\\
&\ll XYZ(X^{-k-1}YZH^2).\end{align*}
The parenthetic factor on the right hand side is at most $X^{-2+2/k}H^{2+1/k}\ll 1$. 
This completes the proof of the lemma.
\end{proof}

Write
$$S_1=r_{\grM^\dagger}(n;Z),\quad S_2=r_\grC(n;Z)-\rho_1(n;Z),$$
and
$$S_3=\rho_1(n;Z)-\rho_2(n;Z),\quad S_4=\rho_2(n;Z)-k^{-1}n^{-1+1/k}YZ.$$
Then since
$$r_\grM(n;Z)-k^{-1}n^{-1+1/k}YZ=S_1+\ldots +S_4,$$
an application of the 
elementary inequality $|S_1+\ldots +S_4|^2\le |S_1|^2+\ldots +|S_4|^2$ combines with 
(\ref{3.8}), (\ref{3.9}), and Lemmata \ref{lemma3.4} and \ref{lemma3.5} to give
\begin{equation}\label{3.11}
\Ups_\grM\ll X^{1+\eps}YZ.
\end{equation}

\section{Minor arcs with a difference} We now estimate $\Ups_\grm$, noting that by 
Bessel's inequality, one has
\begin{equation}\label{4.1}
\Ups_\grm\le \int_\grm |fgu|^2\d\alp =T-\int_\grM |fgu|^2\d\alp ,
\end{equation}
in which
$$T=\int_0^1|fgu|^2\d\alp .$$
By orthogonality, the mean value $T$ counts the number of integral solutions of the 
equation (\ref{1.4}) with $X<x_i\le 2X$, $Y<y_i\le 2Y$ and $1\le z_i\le Z$ for $i=1,2$. 
Put $h=x_1-x_2$, and for concision write $x=x_2$. Then the equation (\ref{1.4}) 
becomes
\begin{equation}\label{4.2}
h\Psi(x,h)=y_1^k-y_2^k+z_1-z_2,
\end{equation}
where
$$\Psi(x,h)=\sum_{j=1}^k\binom{k}{j}x^{k-j}h^{j-1}.$$
For any solution of (\ref{4.2}) counted by $T$, we have
$$|h|\le X^{1-k}((2^k-1)Y^k+Z)\le H.$$  
Thus, on putting
\begin{equation}\label{4.3}
F(\alp)=\sum_{|h|\le H}\sum_{\substack{X<x\le 2X\\ X<x+h\le 2X}}e(h\Psi(x,h)\alp),
\end{equation}
we infer via orthogonality that
$$T=\int_0^1F(\alp)|g(\alp)u(\alp)|^2\d\alp .$$
In view of (\ref{4.1}), therefore, we obtain the relation
\begin{equation}\label{4.4}
\Ups_\grm\le \int_0^1F|g^2u^2|\d\alp -\int_\grM |fgu|^2\d\alp .
\end{equation}

\par We require a modified Hardy-Littlewood dissection for the discussion of the mean 
value $T$. Put $C=k^{-3k}$, and let $\grN$ denote the union of the 
intervals 
$$\grN(q,a)=\{ \alp\in [0,1):|q\alp-a|\le CXY^{-k}\},$$
with $0\le a\le q\le X$ and $(a,q)=1$. Also, we denote by $\grN^\dagger$ the 
corresponding union of the intervals $\grN(q,a)$ in which $q>1$. Further, we put 
$\grn=[0,1)\setminus \grN$.

\begin{lemma}\label{lemma4.1} One has
$$\int_\grn|Fg^2u^2|\d\alp \ll X^{1+\eps}YZ.$$
\end{lemma}

\begin{proof} Suppose that $\alp\in \dbR$, $a\in \dbZ$ and $q\in \dbN$ satisfy 
$(a,q)=1$  and $|\alp-a/q|\le q^{-2}$. Then it follows from a pedestrian generalisation of 
the proof of \cite[Lemma 1]{Vau1986b} with $\nu=\sig$ that, when $4\le k\le 7$, one has
\begin{equation}\label{4.6}
F(\alp)\ll X^{1+\sig+\eps}(q^{-1}+X^{-1}+qX^{1-k-\sig})^{2^{2-k}}.
\end{equation}
Here, we have observed that the term with $h=0$ in (\ref{4.3}) contributes $O(X)$ to 
$|F(\alp)|$, this being majorised by the term $X^{-1}$ in the parenthetic expression on 
the right hand side of (\ref{4.6}), since $\sig=2^{2-k}$ for $4\le k\le 7$. The same 
conclusion follows from the proof of the lemma of \cite{Vau1985} in the case $k=3$.\par

Let $\alp\in \grn$. An application of Dirichlet's theorem on Diophantine approximation 
shows that there exist $a\in \dbZ$ and $q\in \dbN$, with $0\le a\le q\le (CX)^{-1}Y^k$ 
and $(a,q)=1$, for which $|q\alp-a|\le CXY^{-k}$. In such circumstances, the definition of 
$\grN$ shows that $q>X$, and hence (\ref{4.6}) yields the bound
$$F(\alp)\ll X^{1+\sig+\eps}(X^{-1}+Y^kX^{-k-\sig})^\sig \ll X^{1+\eps}.$$

\par When $k\ge 8$, meanwhile, we apply the method of proof of 
\cite[Lemma 10.3]{VW1991} in which we formally take $M=\tfrac{1}{2}$ and $R=2$. By 
substituting the conclusion of \cite[Theorem 1.5]{Woo2012}, in the enhanced form made 
available via \cite[Theorem 1.2]{Woo2014}, for \cite[Lemma 10.2]{VW1991}, one finds 
that the bound
$$\sup_{\alp\in \grn}|F(\alp)|\ll X^{1-\sig+\eps}H$$
holds with $\sig=(2(k-2)(k-3))^{-1}$. Hence, when $\alp\in \grn$, one has 
$F(\alp)\ll X^{1+\eps}$ in all cases. We note that both here, in considering the exponents 
$k\ge 8$, and in our earlier treatment for $3\le k\le 7$, the exponential sum $F(\alp)$ 
differs from the analogues occurring in the cited sources only by the presence of the 
additional summation condition $X<x+h\le 2X$ in (\ref{4.3}). However, the latter is easily 
accommodated in the respective proofs of the desired conclusions.\par

On recalling (\ref{3.4}) and Lemma \ref{lemma3.0}, we now see that
$$\int_\grn|Fg^2u^2|\d\alp \ll X^{1+\eps}\int_0^1|gu|^2\d\alp \ll X^{1+\eps}YZ.$$
This completes the proof of the lemma.
\end{proof}

It is convenient to isolate the diagonal contribution within $F(\alp)$. Write
\begin{equation}\label{4.X}
F_1(\alp)=\sum_{1\le h\le H}
\sum_{\substack{X<x\le 2X\\ X<x+h\le 2X}}e(h\Psi(x,h)\alp),
\end{equation}
and observe that, in view of (\ref{4.3}), one then has
\begin{equation}\label{4.Y}
F(\alp)=2\,\text{Re}\,F_1(\alp)+O(X).
\end{equation}

\begin{lemma}\label{lemma4.2}
One has
$$\int_{\grN^\dagger}|Fg^2u^2|\d\alp \ll X^{1+\eps}YZ.$$
\end{lemma}

\begin{proof} On recalling (\ref{3.4}) and the estimate supplied by Lemma 
\ref{lemma3.0}, one finds that (\ref{4.Y}) yields the relation
\begin{equation}\label{4.8}
\int_{\grN^\dagger}|Fg^2u^2|\d\alp \ll XI_1+\int_{\grN^\dagger}|F_1g^2u^2|\d\alp .
\end{equation}
By reference to the argument leading to (\ref{3.2a}), we find that when $a\in \dbZ$, 
$q\in \dbN$ and $\bet+a/q\in \grN(q,a)\subseteq \grN^\dagger$, one has
\begin{equation}\label{4.9}
u(\bet+a/q)\ll q+ZXY^{-k}\ll X.
\end{equation}

\par Suppose first that $k\ge 4$. Then an application of Schwarz's inequality in 
combination with Lemma \ref{lemma3.0} reveals that
$$\int_{\grN^\dagger}|Fg^2u^2|\d\alp \ll XYZ+X^2I_2^{1/2}I_3^{1/2},$$
where
\begin{equation}\label{4.9a}
I_2=\int_0^1|F_1(\alp)|^2\d\alp \quad \text{and}\quad I_3=\int_0^1|g(\alp)|^4\d\alp .
\end{equation}
By orthogonality, the integral $I_2$ counts the number of integral solutions of the equation 
$h_1\Psi(x_1,h_1)=h_2\Psi(x_2,h_2)$, with $X<x_i\le 2X$ and $1\le h_i\le H$ for $i=1,2$. 
A divisor function estimate confirms that, for each fixed choice of $x_2$ and $h_2$, there 
are $O((XH)^\eps)$ possible choices for $x_1$ and $h_1$, whence 
$I_2\ll (XH)^{1+\eps}$. Meanwhile, the bound $I_3\ll Y^{2+\eps}$ follows from Hua's lemma (see \cite[Lemma 2.5]{Vau1997}). Hence
\begin{align*}
\int_{\grN^\dagger}|Fg^2u^2|\d\alp &\ll XYZ+X^{2+\eps}(XH)^{1/2}Y\ll 
X^{1+\eps}YZ(1+X^{3/2}Z^{-1}H^{1/2}).
\end{align*}
Since $X^{-k+7/2-1/k}H^{1/2+1/k}\le 1$, the conclusion of the lemma follows for 
$k\ge 4$.\par

We turn next to the situation in which $k=3$. Put $A=Z^{-1/2}X^{-1/8}$, and divide the 
set $\grN^\dagger$ into the two subsets
$$\grN^\dagger_0=\{ \alp \in \grN^\dagger:\|\alp\|\le A\}\quad \text{and}\quad 
\grN_1^\dagger =\{\alp\in \grN^\dagger:\|\alp\|>A\}.$$
Making use of the familiar estimate $u(\alp)\ll \|\alp\|^{-1}$, we find that
$$\int_{\grN_1^\dagger}|F_1g^2u^2|\d\alp \ll (Z^{1/2}X^{1/8})^2
\int_{\grN_1^\dagger}|F_1g^2|\d\alp .$$
An application of Schwarz's inequality yields the bound
$$\int_{\grN_1^\dagger}|F_1g^2|\d\alp \ll I_2^{1/2}I_3^{1/2},$$
where $I_2$ and $I_3$ are defined as in (\ref{4.9a}). We observe that our earlier bounds 
for $I_2$ and $I_3$ remain valid also when $k=3$. Thus, we conclude that
\begin{align}
\int_{\grN^\dagger_1}|F_1g^2u^2|\d\alp &\ll ZX^{1/4+\eps}(XH)^{1/2}(Y^2)^{1/2}
\ll X^{1+\eps}YZ(X^{-1/4}H^{1/2}).\label{4.9b}
\end{align}

\par For the treatment of $\grN_0^\dagger$, we require a sharp upper bound for
$$I_4=\int_{\grN_0^\dagger}|g(\alp)|^4\d\alp .$$
Recall (\ref{3.aa}), and define
$$g^*(\alp)=q^{-1}S(q,a)w(\alp-a/q),$$
when $\alp\in \grN(q,a)\subseteq \grN$, and otherwise set $g^*(\alp)=0$. Then we find 
from \cite[Theorem 4.1]{Vau1997} that whenever $\alp\in \grN$, one has 
$g(\alp)-g^*(\alp)\ll X^{1/2+\eps}$. Hence
\begin{equation}\label{4.11}
I_4\ll \int_\grN|g^*(\alp)|^4\d\alp +X^{2+\eps}\text{mes}(\grN_0^\dagger).
\end{equation}
From \cite[Lemmata 4.9 and 6.2]{Vau1997}, one readily infers the bound
$$\int_\grN |g^*(\alp)|^4\d\alp \ll Y^{1+\eps}.$$
Meanwhile
\begin{align*}
\text{mes}(\grN_0^\dagger)&\le \sum_{1\le q\le X}\sum_{\substack{1\le a\le q\\ 
\|a/q\|\le 2Z^{-1/2}X^{-1/8}}}\text{mes}(\grN(q,a))\\
&\ll \sum_{1\le q\le X}\left( qZ^{-1/2}X^{-1/8}\right) (q^{-1}XY^{-3})\ll 
X^{15/8}Y^{-3}Z^{-1/2}.
\end{align*}
On substituting these estimates into (\ref{4.11}), we discern that
\begin{equation}\label{4.12}
I_4\ll Y^{1+\eps}+X^{31/8+\eps}Y^{-3}Z^{-1/2}\ll Y^{1+\eps},
\end{equation}
since $\frac{31}{8}-3\left(\frac{5}{6}\right)-\frac{1}{2}\left(\frac{7}{6}\right)
=\frac{19}{24}<\frac{5}{6}$.\par

Next, by (\ref{4.9}) and the inequalities of Cauchy and Schwarz, one has
\begin{equation}\label{4.13}
\int_{\grN_0^\dagger}|F_1g^2u^2|\d\alp \ll XI_4^{1/2}(HI_5)^{1/2},
\end{equation}
where
$$I_5=\int_0^1F_2(\alp)|u(\alp)|^2\d\alp ,$$
in which we write
$$F_2(\alp)=\sum_{1\le h\le H}\biggl| \sum_{\substack{X<x\le 2X\\ X<x+h\le 2X}}
e(h\Psi(x,h)\alp)\biggr|^2.$$
The integral $I_5$ does not exceed the number of integral solutions of the equation
$$h(\Psi(x_1,h)-\Psi(x_2,h))=z_1-z_2,$$
with $1\le h\le H$, $X<x_1,x_2\le 2X$ and $1\le z_1,z_2\le Z$. Since $x_1-x_2$ divides 
the polynomial $\Psi(x_1,h)-\Psi(x_2,h)$, it follows via an elementary divisor function 
estimate that, whenever $z_1$ and $z_2$ are fixed with $z_1\ne z_2$, then there are 
$O(Z^\eps)$ possible choices for $h$, $x_1$ and $x_2$. Hence we deduce that
$$I_5\ll HXZ+Z^{2+\eps}\ll HXZ.$$
On substituting this bound together with (\ref{4.12}) into (\ref{4.13}), we see that
$$\int_{\grN_0^\dagger}|F_1g^2u^2|\d\alp \ll XY^{1/2}(H^2XZ)^{1/2}.$$
This, in combination with Lemma \ref{lemma3.0} and equations (\ref{4.8}) and 
(\ref{4.9b}), gives
$$\int_{\grN^\dagger}|Fg^2u^2|\d\alp 
\ll X^{1+\eps}YZ(1+X^{-1/4}H^{1/2}+X^{1/2}Y^{-1/2}Z^{-1/2}H).$$
Since $\frac{1}{2}-\frac{5}{12}-\frac{7}{12}+\frac{1}{2}=0$, the conclusion of the 
lemma follows for $k=3$.
\end{proof}

The treatment of the minor arcs is now coming to an end. Define
$$\grD=\{\alp \in [0,1):\|\alp\|\le CXY^{-k}\}.$$
Note that $\grM=\grM^\dagger\cup \grC$ and $\grN=\grN^\dagger\cup \grD$. 
Since $[0,1)=\grD\cup \grN^\dagger\cup \grn$, it follows by combining Lemmata 
\ref{lemma4.1} and \ref{lemma4.2} that
$$\int_0^1F|gu|^2\d\alp =\int_\grD F|gu|^2\d\alp +O(X^{1+\eps}YZ).$$
Likewise, we obtain from Lemma \ref{lemma3.3} the relation
$$\int_\grM|fgu|^2\d\alp =\int_\grC|fgu|^2\d\alp +O(X^{1+\eps}YZ).$$
Hence, we conclude from (\ref{4.4}) that
\begin{equation}\label{4.Z}
\Ups_\grm \le \int_\grD F|gu|^2\d\alp -\int_\grC|fgu|^2\d\alp +O(X^{1+\eps}YZ).
\end{equation}

\section{The annihilation of the central intervals}
In this penultimate section, we complete the estimation of $\Ups_{\grm}$ by exploiting 
cancellations between the two integrals on the right hand side of equation (\ref{4.Z}). 
With this in view, we put $\grc=\grD\setminus\grC$ and recast the relation (\ref{4.Z}) as
\begin{equation}\label{5.1}
\Ups_\grm \le \int_\grC  (F-|f|^2)|gu|^2\d \alp+\int_\grc F|gu|^2\d \alp + 
O(X^{1+\eps}YZ).
\end{equation}

We first show that the integral over $\grc$ can be absorbed into the error term. The 
argument will depend on the following simple estimate.

\begin{lemma}\label{L5.1}
Let $\Delta$ be a positive number. Then
$$ \int_{-\Del}^\Del |g(\alp)|^2\d\alp \ll \Del Y + Y^{2-k+\eps}.$$
\end{lemma}

\begin{proof} By (\ref{2.6}), one has
$$ \int_{-\Del}^\Del |g(\alp)|^2\d\alp = \sum_{Y<y_1,y_2\le 2Y}  \int_{-\Del}^\Del
e(\alp(y_1^k-y_2^k))\d\alp. $$
The terms with $y_1=y_2$ contribute $2\Del Y$. The remaining terms contribute an 
amount not exceeding
$$ \sum_{\substack{Y<y_1,y_2\le 2Y\\y_1\neq y_2}} \frac{2}{|y_1^k-y_2^k|}.$$
Here, we write $l=y_1-y_2$, and observe that by symmetry, it suffices to estimate the 
part of the sum where $l>0$. But then $y_1^k-y_2^k\gg lY^{k-1}$, and the sum in the 
preceding display is therefore bounded by
$$ \sum_{1\le l\le Y} \sum_{Y<y_2\le 2Y} \frac{1}{lY^{k-1}} \ll Y^{2-k+\eps}.$$ 
The desired conclusion now follows.
\end{proof}

\begin{lemma}\label{L5.2} One has
$$\int_\grc F(\alp)|g(\alp)u(\alp)|^2\d\alp \ll X^{1+\eps}YZ.$$
\end{lemma}

\begin{proof}
We note that when $\alp\in\grD$ one has
$$ HX^{k-2} \| \alp\|  \le 2^kCX^{\sig+k-2} XY^{-k} \le 2^kC. $$
Hence, temporarily assuming that $k\ge 4$ and estimating the sum $F_1(\alp)$ defined in 
(\ref{4.X}) via \cite[Lemma 2]{Vau1986b}, we first deduce that
$$ F_1(\alp) \ll HX( 1+ HX^{k-1} \|\alp\|)^{-1} + H, $$
and then infer from (\ref{4.Y}) the bound
\begin{equation}\label{5.2}
F(\alp) \ll HX( 1+ HX^{k-1} \|\alp\|)^{-1} + X.
\end{equation}
The proof of \cite[Lemma 2]{Vau1986b} remains valid when $k=3$ and $q=1$ (in the 
notation of this reference). Hence (\ref{5.2}) holds for all $k\ge 3$, and consequently,
\begin{equation} \label{5.3}
\int_\grc  F|gu|^2\d\alp \ll X I_1 + \Gam ,
\end{equation}
where $I_1$ is given by (\ref{3.4}), and 
$$ \Gam = HX \int_{\grc} \frac{|g(\alp)u(\alp)|^2}{1+HX^{k-1}\|\alp\|}\d\alp. $$
Note that $\grc$ is the union of two intervals, one of which being $[QX^{-k}, XY^{-k}]$.
By symmetry, and since the integrand has period 1, it suffices to estimate the contribution 
from this interval. This we cover by $O(\log X)$ disjoint intervals $[AY^{-k}, 2AY^{-k}]$, 
with $QX^{\sig-1}\le A \le X$. By Lemma \ref{L5.1}, making use of the trivial bound 
$|u(\alp)|\le Z$, we find that
\begin{align*}
\int_{AY^{-k}}^{2AY^{-k}}\frac{|g(\alp)u(\alp)|^2}{1+HX^{k-1}\alp}\d\alp &\ll 
Z^2H^{-1} X^{1-k} A^{-1} Y^k  \int_{AY^{-k}}^{2AY^{-k}} |g(\alp)|^2\d\alp \\
&\ll Z^2A^{-1} (AY^{1-k} + Y^{2-k+\eps}) \\
&\ll Z^2 Y^{1-k} + Z^2Q^{-1}X^{1-\sig} Y^{2-k+\eps}.
\end{align*}
Here the second term on the right hand side dominates, and we infer the bound
$$ \Gam \ll HX^{2-\sig} Y^{2-k+\eps} Z^2 Q^{-1}
\ll X^{1+\eps}YZ(XY^{1-k}Q^{-1} Z). $$
Since $Y/Q = X^{(2\sig-1)/k}$ and $XY^{-k}Z \ll X^{-\sig +1/k}$, it follows that 
$\Gam \ll X^{1+\eps}YZ$. The lemma now follows from (\ref{5.3}) and Lemma 
\ref{lemma3.0}.
\end{proof}

\begin{lemma}\label{L5.3} One has
$$ \int_\grC (F-|f|^2) |gu|^2 \d\alp = \int_\grC (F-|f|^2) |wu|^2 \d\alp + O(XYZ). $$
\end{lemma}

\begin{proof} When $\alp\in\grC$, we find from \cite[Theorem 4.1]{Vau1997} that
$g(\alp)=w(\alp)+O(1)$, and hence $|g(\alp)|^2 = |w(\alp)|^2 + O(|w(\alp)|)$. On multiplying 
this relation with $(F-|f|^2) |u|^2$, one finds that the lemma will follow from the estimate
\begin{equation}\label{5.4}
 \int_\grC |(F-|f|^2)wu^2| \d\alp \ll XYZ,
\end{equation}
that we now establish in  two steps.

First we observe that \cite[Lemma 6.2]{Vau1997} delivers the bound
$$ \int_\grC |w(\alp)|\d\alp \ll Y \int_{-1/2}^{1/2} (1+Y^k|\alp|)^{-1}\d\alp \ll Y^{1-k+\eps}. $$
Hence, the trivial bounds $F(\alp)\ll HX$ and $u(\alp)\ll Z$ suffice to conclude that
\begin{equation}\label{5.5}
\int_\grC |Fwu^2| \d\alp \ll HXY^{1-k+\eps}Z^2 = XYZ(HY^{\eps-k}Z), 
\end{equation}
and we note that $HY^{\eps-k}Z\ll 1$.

Another appeal to  \cite[Theorem 4.1]{Vau1997} shows that whenever $\alp\in\grC$, one has 
$f(\alp) = v(\alp) + O(1)$, and \cite[Lemma 6.2]{Vau1997} then delivers the estimate
$$ f(\alp) \ll X(1+ X^k\|\alp\|)^{-1}. $$
Using trivial bounds for $w(\alp)$ and $u(\alp)$, we now infer that
$$ \int_\grC  |f^2wu^2| \d\alp \ll YZ^2X^2 \int_{-1/2}^{1/2} (1+X^k|\alp|)^{-2}\d\alp
\ll YZ^2X^{2-k} \ll XYZ. $$
On combining this bound with (\ref{5.5}), we arrive at (\ref{5.4}). This completes the proof of the lemma.
\end{proof}

\begin{lemma}\label{L5.4} Let $\grK=[0,1]\setminus\grC$. Then
$$\int_\grK (F(\alp)-|f(\alp)|^2) |w(\alp)u(\alp)|^2 \d\alp \ll XYZ.$$
\end{lemma}

\begin{proof} The argument is similar to the one used to demonstrate the previous lemma.
We again use \cite[Lemma 6.2]{Vau1997}, this time providing the bound
\begin{align} \label{5.6}
\int_\grK  |w(\alp)|^2 \d\alp &\ll Y^2 \int_{Q/X^k}^{1/2} (1+Y^k\alp)^{-2}\d\alp
\notag \\ & \ll Y^{2-k} Q^{-1} X^kY^{-k} \ll Y^{2-k} H^{-1+1/k}.
\end{align}
The trivial bound for $F(\alp)|u(\alp)|^2$ now implies that
\begin{equation} \label{5.7}
\int_\grK F|wu|^2\d\alp \ll  Y^{2-k} H^{-1+1/k} HXZ^2 \ll XYZ,
\end{equation}
because one has $Y^{1-k}H^{1/k}Z\ll H^{-1+2/k} \ll 1$.

More care is required for the term involving $|f(\alp)|^2$. Here, we split $\grK$ into its 
subsets $\grc$ and $\grK\setminus\grc = \{\alp\in[0,1]: \|\alp\|>CXY^{-k}\}$. The argument leading to (\ref{5.6}) yields 
$$ \int_{\grK\setminus\grc} |w(\alp)|^2\d\alp \ll Y^{2-k} X^{-1}, $$
so that a trivial bound for $|f(\alp)u(\alp)|^2$ provides the estimate
\begin{equation}
\label{5.8}\int_{\grK\setminus\grc} |fwu|^2\d\alp \ll XY^{2-k}Z^2 \ll XYZ.
\end{equation}

It remains to examine the contribution from $\grc$. For $\alp\in\grc$ we deduce from 
\cite[Theorem 4.1 and Lemma 6.2]{Vau1997} that
$$ f(\alp) \ll X(1+X^k\|\alp\|)^{-1} + (X^k\|\alp\|)^{1/2}, $$
and hence,
$$  |f(\alp)w(\alp)|^2 \ll X^2Y^2(1+X^k\|\alp\|)^{-2} + 
X^kY^2\|\alp\| (1+Y^k\|\alp\|)^{-2}. $$
Since $\|\alp\| \ge QX^{-k}\gg H^{1-1/k}Y^{-k}$, the previous bound implies that
$$ |f(\alp)w(\alp)|^2 \ll X^2Y^2Q^{-1} (1+X^k\|\alp\|)^{-1} + 
X^kY^{2-2k}\|\alp\|^{-1} . $$
By applying a trivial bound for $u(\alp)$, we may conclude that
\begin{align}\label{5.9}
\int_\grc |fwu|^2\d\alp &\ll Z^2 ( X^{2-k+\eps} Y^2 Q^{-1} + X^{k+\eps}Y^{2-2k})
\notag\\ &\ll XYZ(H^{2/k} X^{\eps-1}+  H^{-2+1/k}).
\end{align}
The lemma now follows from (\ref{5.7}), (\ref{5.8}) and (\ref{5.9}).
\end{proof}

We are ready to assemble the puzzle. By combining Lemmata \ref{L5.2}, \ref{L5.3} and 
\ref{L5.4}, we find from (\ref{5.1}) that
$$ \Ups_\grm \le \int_0^1 (F-|f|^2)|wu|^2 \d\alp + O(X^{1+\eps} YZ).$$
By applying orthogonality and reversing the transformation $h=x_1-x_2$ and $x=x_2$ 
within (\ref{4.3}), one finds that the main term here is a weighted count of the integral 
solutions of the equation
$$x_1^k-x_2^k=m_1-m_2+z_1-z_2,$$ 
with $X<x_i\le 2X$, $Y^k<m_i\le (2Y)^k$ and $1\le z_i\le Z$ $(i=1,2)$, subject to the 
condition $|x_1-x_2|>H$. For 
each such putative solution, one has
$$|x_1^k-x_2^k|\ge kHX^{k-1}>(2Y)^k+Z>|m_1-m_2+z_1-z_2|,$$
whence one infers that in fact no solutions exist. Thus we conclude that the contribution 
of $F$ to $\Ups_\grm$ annihilates the anti-contribution of $|f|^2$, implying that
$\Ups_\grm \ll X^{1+\eps}YZ$. By combining this estimate with (\ref{3.11}) and 
(\ref{2.9}), we arrive at the bound
\begin{equation}\label{5.Z}  \Ups(N,Z)\ll X^{1+\eps}YZ.
\end{equation}

\section{Deduction of the main results}

\begin{proof}[The proof of Theorem {\rm \ref{theorem1.1}}] Recall that 
$\phi_k=(1-1/k)^2$, and that $E_k(N,Z)$ denotes the number of integers $n$ with 
$N<n\le 2N$ for which the interval $(n,n+Z]$ contains no integer that is the sum of two 
positive integral $k$th powers. For the latter integers $n$, one has $r(n;Z)=0$. Therefore, 
when $N^{\tet_k}\le Z\le 2k^2N^{\phi_k}$, it follows from (\ref{2.5}) and (\ref{5.Z}) 
that
$$ E_k(N,Z) \left(k^{-1}N^{-1+1/k} YZ\right)^2 \le \Ups(N,Z)\ll X^{1+\eps}YZ, $$
whence 
$$ E_k(N,Z) \ll N^{2-2/k} X^{1+\eps} (YZ)^{-1} \ll N^{1+\tet_k+\eps}Z^{-1}. $$
When $Z>2k^2N^{\phi_k}$, meanwhile, it follows via the greedy algorithm that 
$E_k(N,Z)=0$ for large $N$. This completes the proof of Theorem \ref{theorem1.1}.
\end{proof}

\begin{proof}[The proof of Theorem {\rm \ref{theorem1.2}}] Within this proof we 
abbreviate $s_{k,n}$ to $s_n$. For large $N$, it follows from (\ref{1.4a}) that whenever 
$s_{n+1}\le N$, then $s_{n+1}-s_n\le k^2N^{\phi_k}$. This shows that $E_k(N,Z) =0$ 
whenever $Z >2k^2N^{\phi_k}$. Let
$$\Xi (N,Z) = \text{card}\{N/2<s_n\le N: Z/2 < s_{n+1}-s_n\le Z\},$$
and put $Z_0= 4k^2N^{\phi_k}$. Then we have $\Xi(N,Z) =0$ for $Z>Z_0$. Also, when 
$Z$ is an even integer with $4\le Z\le Z_0$ and $s_{n+1}-s_n>Z$, then each of the 
intervals $(s_n+m-1,s_n+m+Z/2)$ $(1\le m\le Z/2)$ contains no sum 
of two positive integral $k$-th powers. Hence
$$ E_k(N,Z/2) \ge (Z/2)\Xi(N,2Z), $$
and therefore, we deduce from Theorem \ref{theorem1.1} that
$$ \Xi(N,2Z) \ll Z^{-1} E_k(N,Z/2) \ll N^{1+\tet_k+\eps}Z^{-2}. $$
We now conclude that
$$ \sum_{N/2<s_n\le N} (s_{n+1}-s_n)^2 \ll \sum_{\substack{j=0\\ 2^j\le Z_0}}^\infty
(2^{-j}Z_0)^2 E_k(N,2^{-j}Z_0) \ll N^{1+\tet_k+2\eps}.
$$
On summing over dyadic intervals, the conclusion of Theorem \ref{theorem1.2} follows.
\end{proof}

\bibliographystyle{amsbracket}

\begin{thebibliography}{18}

\bibitem{Bro2002}
T. D. Browning, \emph{Equal sums of two $k$th powers}, J. Number Theory \textbf{96} 
(2002), no. 2, 293--318. 

\bibitem{Bru2001}
J. Br\"udern, \emph{Cubic Diophantine inequalities III}, Period. Math. Hungar. \textbf{42} 
(2001), no. 1-2, 211--226.

\bibitem{BW2004}
J. Br\"udern and T. D. Wooley, \emph{Additive representation in short intervals, I: Waring's 
problem for cubes}, Compos. Math. \textbf{140} (2004), no. 5, 1197--1220.

\bibitem{Dan1997}
S. Daniel, \emph{On gaps between numbers that are sums of three cubes}, Mathematika 
\textbf{44} (1997), no. 1, 1--13. 

\bibitem{Erd1939}
P. Erd\H os, \emph{On the integers of the form $x^k+y^k$}, J. London Math. Soc. 
\textbf{14} (1939), 250--254.

\bibitem{EM1938}
P. Erd\H os and K. Mahler, \emph{On the number of integers that can be represented by a 
binary form}, J. London Math. Soc. \textbf{13} (1938), 134--139.

\bibitem{Fri1982}
J. B. Friedlander, \emph{Sifting short intervals}, Math. Proc. Cambridge Philos. Soc. 
\textbf{91} (1982), no. 1, 9--15.

\bibitem{Gre1966}
G. Greaves, \emph{On the representation of a number as a sum of two fourth powers}, 
Math. Z. \textbf{94} (1966), 223--234.

\bibitem{Gre1994}
G. Greaves, \emph{Representation of a number by the sum of two fourth powers}, 
Mat. Zametki \textbf{55} (1994), no. 2, 47--58, 188.

\bibitem{Har1991}
G. Harman, \emph{Sums of two squares in short intervals}, Proc. London Math. Soc. (3) 
\textbf{62} (1991), no. 2, 225--241.

\bibitem{HB1997}
D. R. Heath-Brown, \emph{The density of rational points on cubic surfaces}, Acta Arith. 
\textbf{79} (1997), no. 1, 17--30.

\bibitem{HB2002}
D. R. Heath-Brown, \emph{The density of rational points on curves and surfaces}, Ann. 
of Math. (2) \textbf{155} (2002), no. 2, 553--595.

\bibitem{Hoo1963}
C. Hooley, \emph{On the representations of a number as the sum of two cubes}, Math. 
Z. \textbf{82} (1963), 259--266.

\bibitem{Hoo1964}
C. Hooley, \emph{On the representation of a number as the sum of two $h$-th powers}, 
Math. Z. \textbf{84} (1964), 126--136.

\bibitem{Hoo1981a}
C. Hooley, \emph{On the numbers that are representable as the sum of two cubes}, J. 
reine angew. Math. \textbf{314} (1980), 146--173.

\bibitem{Hoo1981b}
C. Hooley, \emph{On another sieve method and the numbers that are a sum of two 
$h$th powers}, Proc. London Math. Soc. (3) \textbf{43} (1981), no. 1, 73--109.

\bibitem{Hoo1994}
C. Hooley, \emph{On the intervals between numbers that are sums of two squares: IV}, 
J. reine angew. Math. \textbf{452} (1994), 79--109.

\bibitem{Hoo1996}
C. Hooley, \emph{On another sieve method and the numbers that are a sum of 
two $h$th powers: II}, J. reine angew. Math. \textbf{475} (1996), 55--75.

\bibitem{Pla1987}
V. A. Plaksin, \emph{The distribution of numbers that can be represented as the sum of 
two squares}, Izv. Akad. Nauk SSSR Ser. Mat. \textbf{51} (1987), no. 4, 860--877, 911.

\bibitem{Pla1992}
V. A. Plaksin, \emph{Letter to the editors: ``The distribution of numbers that can be 
represented as the sum of two squares'' [Izv. Akad. Nauk SSSR Ser. Mat. 51 (1987), no. 4, 
860--877, 911]}, Izv. Ross. Akad. Nauk Ser. Mat. \textbf{56} (1992), no. 4, 908--909.

\bibitem{Sal2008}
P. Salberger, \emph{Rational points of bounded height on projective surfaces}, Math. Z. 
\textbf{258} (2008), no. 4, 805--826.

\bibitem{SW1995}
C. M. Skinner and T. D. Wooley, \emph{Sums of two $k$th powers}, J. reine angew. Math. 
\text{462} (1995), 57--68. 

\bibitem{Vau1985}
R. C. Vaughan, \emph{Sums of three cubes}, Bull. London Math. Soc. \textbf{17} 
(1985), no. 1, 17--20.

\bibitem{Vau1986b}
R. C. Vaughan, \emph{On Waring's problem for smaller exponents}, Proc. London Math. 
Soc. (3) \textbf{52} (1986), no. 3, 445--463.

\bibitem{Vau1986c}
R. C. Vaughan, \emph{On Waring's problem for sixth powers}, J. London Math. Soc. (2) 
\textbf{33} (1986), no. 2, 227--236.

\bibitem{Vau1997}
R. C. Vaughan, \emph{The Hardy-Littlewood method}, 2nd edition, Cambridge University Press, Cambridge, 1997.

\bibitem{VW1991}
R. C. Vaughan and T. D. Wooley, \emph{On Waring's problem: some refinements}, Proc. 
London Math. Soc. (3) \textbf{63} (1991), no. 1, 35--68

\bibitem{Woo1995}
T. D. Wooley, \emph{Sums of two cubes}, Internat. Math. Res. Notices (1995), no. 4, 
181--184.

\bibitem{Woo2012}
T. D. Wooley, \emph{Vinogradov's mean value theorem via efficient congruencing}, 
Ann. of Math. (2) \textbf{175} (2012), no. 3, 1575--1627.

\bibitem{Woo2014}
T. D. Wooley, \emph{The cubic case of the main conjecture in Vinogradov's mean value 
theorem}, submitted; arXiv:1401.3150.


\end{thebibliography}
\providecommand{\bysame}{\leavevmode\hbox to3em{\hrulefill}\thinspace}

\end{document}